\documentclass[11pt]{amsart}
\pdfoutput=1

\usepackage[a4paper,bottom=35mm]{geometry}
\usepackage[dvipsnames]{xcolor}
\usepackage{amsmath,amssymb}

\usepackage{url}
\usepackage{graphicx}
\usepackage{hyperref}
\usepackage{multirow}

\newtheorem{thm}{Theorem}[section]
\newtheorem{prop}[thm]{Proposition}

\newtheorem{rmk}[thm]{Remark}
\newtheorem{example}[thm]{Example}
\newtheorem{ques}[thm]{Question}

\newcommand{\E}{\mathcal{E}}
\newcommand{\N}{\mathbb{N}}
\newcommand{\T}{\mathcal{T}}
\newcommand{\bd}{\partial}
\newcommand{\Z}{\mathbb{Z}}

\newcommand{\C}{\mathbb{C}}

\newcommand{\tet}{I_{\Delta}}

\newcommand{\qsin}{\mathcal{S}_q}
\newcommand{\qcos}{\mathcal{C}_q}

\author{Daniele Celoria}
\title{A note on the tetrahedral index and the Hahn-Exton $q$-Bessel function}

\begin{document}
\begin{abstract}
The purpose of this short note is twofold: First to elucidate some connections between the ``building block'' of Dimofte--Gaiotto--Gukov's $3$D index, known as the tetrahedral index $\tet\left(m,e\right)$, and Hahn--Exton's $q$-analogue of the Bessel function $J_\nu \left(z;q\right)$. The correspondence between $\tet$ and $J_\nu$ will allow us to translate useful relations from one setting to the other.
Second, we want to introduce to the $q$-hypergeometric community some possibly new techniques, theory and conjectures arising from applications of physical mathematics to geometric topology.
\end{abstract}
\maketitle

\section{Introduction}

The study of quantum invariants of $3$-manifolds has seen remarkable developments over the past two decades, partly driven by new results and connections with geometric topology. Among the most interesting constructions to emerge from this interplay is the $3$D index, introduced by Dimofte, Gaiotto, and Gukov \cite{dimofte2014gauge,dimofte20133}, which assigns a formal power series $I_\T\left(q\right) \in \Z[\![q^{\frac12}]\!]$ to an ideal triangulation $\T$ of a cusped $3$-manifold $M$. This invariant, built from a basic building block known as the tetrahedral index $\tet\left(m,e\right)$, encodes deep geometric and topological information on $M$~\cite{garoufalidis20163d,garoufalidis20163d_angle, garoufalidis2022periods, garoufalidis20243d}.

On the other hand, the Hahn–Exton $q$-Bessel function $J_\nu\left(z;q\right)$, a classical object in the theory of basic hypergeometric series, has long been studied for its rich analytic structure and its connections to physics, quantum groups, combinatorics, to mention a few~\cite{SwarttouwPhD, koelink1995q, koelink1994quantum}.

The starting point of this note is to make explicit the correspondence between these two functions, by identifying the tetrahedral index with an evaluation of the Hahn–Exton function. While this connection is  known to experts familiar with both frameworks, it seems to have never appeared in the literature. Our aim is to demonstrate how known identities for $J_\nu\left(z;q\right)$ can be translated into non-trivial relations for the tetrahedral index, and vice versa.

Beyond these structural parallels, we also explore how this correspondence can be used to translate conjectural or experimental observations with a topological flavour into the language of $q$-hypergeometric series. In particular, we revisit a formula of Gang and Yonekura~\cite{gang2018symmetry} expressing the $3$D index of a Dehn filled manifold in terms of the index of the unfilled one~\cite{3dindexdehn}, and suggest that its effectiveness (even in cases where convergence is not guaranteed) may be better understood through number theoretic techniques.

While no fundamentally new theorems are proved here, we hope that this note will be of interest to two distinct communities: to topologists seeking tools for understanding quantum invariants, and to number theorists interested in new applications of classical $q$-series. We also include some conjectural identities and open questions.

\subsection{Acknowledgements} The author wishes to thank Craig Hodgson, Hyam Rubinstein and Ole Warnaar for helpful discussions and sharing their expertise. This research has been partially supported by grant DP190102363 from the Australian
Research Council.

\section{Tetrahedral index and the Hahn-Exton function}

We will be using the standard notation for $q$-hypergeometric functions (also known as basic hypergeometric functions), following Gasper and Rahman's book~\cite{GasperRahman}.\\

Recall the definition of the basic hypergeometric series:
\begin{equation}
_{r}\phi_s \left[\begin{matrix}
a_1 & a_2 & \ldots & a_{r} \\
b_1 & b_2 & \ldots & b_{s} \end{matrix}
; q,z \right] = \sum_{n=0}^\infty
\frac {\left(a_1, a_2, \ldots, a_{r};q\right)_n} {\left(b_1, b_2, \ldots, b_s,q;q\right)_n} \left( \left(-1\right)^n q^{\binom{n}{2} }  \right)^{1+s-r}z^n,
\end{equation}
where $\left(c_1, c_2, \ldots, c_{r};q\right)_n$ is a shorthand for $\prod_{i = 1}^{r} \left(c_i;q\right)_n$, and $\left(c;q\right)_n = \prod_{i=0}^{n-1} \left(1-c  q^i\right)$ is known as a $q$-Pochhammer symbol. Here $n \in \N \cup \infty$.

\subsection{The tetrahedral index}\label{ssec:tet_index}

The ``building block'' of the $3$D index, first introduced in~\cite{dimofte20133}, is the $q$-holonomic function $I_\Delta \colon \Z^2 \longrightarrow \Z(\!(q^{\frac12})\!)$ in two variables $m,e \in \Z$, given explicitly by
\begin{equation}\label{eqn:tetrahedral_index}
I_\Delta \left(m,e\right) = \sum_{n = e_+}^\infty \left(-1\right)^n \frac{q^{\frac{n\left(n+1\right)}{2} - \left(n +\frac{e}{2} \right)m}}{\left(q\right)_n \left(q\right)_{n+e}},
\end{equation}
where $e_+ = \frac{|e| - e}{2}$. For notational convenience, the dependency on $q$ of $\tet$ and in the $q$-Pochhammer symbols has been suppressed.

By~\cite[Sec.~3.1]{dimofte20133}, tetrahedral indices can be identified as the coefficients in the generating function
\begin{equation}\label{eqn:index_charge_basis}
\frac{\left(q^{-\frac{m}{2} +1} \zeta^{-1}\right)_\infty}{\left(-q^{-\frac12} \zeta\right)_\infty} = \sum_{e \in \Z} I_\Delta \left(m,e\right) \zeta^e.
\end{equation}

The left hand side of equation~\eqref{eqn:index_charge_basis} is the product $e_q\left(-q^{-\frac12} \zeta\right) E_q\left(q^{-\frac{m}{2} +1} \zeta^{-1}\right)$ of two well-known $q$-analogues of the exponential function:
\begin{equation}\label{eqn:exp1}
e_q\left(z\right) = \sum_{n \ge 0} \frac{z^n}{\left(q\right)_n} = {_1 \phi_0} \left[ \begin{matrix}
0\\ -
\end{matrix};q,z\right] = \frac{1}{\left(z\right)_\infty},
\end{equation}
\begin{equation}\label{eqn:exp2}
E_q\left(z\right) = \sum_{n \ge 0} \frac{q^{\binom{n}{2}}}{\left(q\right)_n}z^n = {_0\phi_0}\left[ \begin{matrix}
-\\ -
\end{matrix};q,-z\right]= \left(-z\right)_\infty
\end{equation}

The latter two identities follow easily from~\cite[Sec.~1.3]{GasperRahman}.\\~\\
The tetrahedral index is known to satisfy numerous functional equations; see \emph{e.g.}~\cite{garoufalidis20163d_angle}, where $q$-holonomic relations completely determining $\tet(m,e)$ are explored. The first one is a simple chain of symmetries:
\begin{equation}\label{eqn:tetrahedral_index_symmetries}
\tet \left(m,e\right)  = \tet \left(-e,-m\right) =  \left(-\sqrt{q}\right)^{-e} \tet\left(e,-e-m\right) = \left(-\sqrt{q}\right)^{m} \tet\left(-e-m,m\right).
\end{equation}

Next, there are two series of three-term relations, which combine the values of $\tet$ on three adjacent lattice points in $\Z^2$:
\begin{gather}\label{eqn:index_3term_consecutive}
\begin{gathered}
I_\Delta\left(m,e+1\right)  +\left( q^{e+\frac{m}{2}}  - q^{-\frac{m}{2}} -q^{\frac{m}{2}} \right)I_\Delta\left(m,e\right)  + I_\Delta\left(m,e-1\right) = 0\\
I_\Delta\left(m+1,e\right)  +\left( q^{-m-\frac{e}{2}}  - q^{-\frac{e}{2}} -q^{\frac{e}{2}} \right)I_\Delta\left(m,e\right)  + I_\Delta\left(m-1,e\right) = 0,
\end{gathered}
\end{gather}

\begin{gather}\label{eqn:index_3term_adjacent}
\begin{gathered}
q^\frac{e}{2} I_\Delta\left(m+1,e\right)  +q^{-\frac{m}{2}} I_\Delta\left(m,e+1\right)  - I_\Delta\left(m,e\right) = 0\\
q^\frac{e}{2} I_\Delta\left(m-1,e\right)  +q^{-\frac{m}{2}} I_\Delta\left(m,e-1\right)  - I_\Delta\left(m,e\right) = 0.
\end{gathered}
\end{gather}

There are two further important relations. The first can be regarded as being some form of ``weighted orthogonality'' between the infinite arrays $\left(\tet(m,e)\right)_{e \in \Z}$, and is usually referred to as \emph{quadratic identity}:
\begin{equation}\label{eqn:quadratic_identity}
\sum_{e \in \Z} q^e \tet \left(m,e\right) \tet \left(m,e+c\right) = \delta_{c,0}.
\end{equation}
The last relation we need is referred to as \emph{pentagon relation}, and holds for variables $m_1,m_2,x_1,x_2,x_3 \in \Z$:
\begin{gather}\label{eqn:pentagon_identity}
\begin{gathered}
\sum_{e \in \Z} q^e \tet \left(m_1 ,e +x_1\right) \tet \left(m_2 ,e+x_2\right) \tet \left(m_1 +m_2, e + x_3\right) =\\ q^{-x_3} \tet \left(m_1-x_2+x_3,x_1-x_3\right) \tet \left(m_2-x_1+x_3, x_2-x_3\right).
\end{gathered}
\end{gather}

The quadratic and pentagon relations are crucial in the partial proof of the invariance of the $3$D index (see Section~\ref{sec:3Dindex} for a definition) under $2$-$3$ and $0$-$2$ moves on triangulations of hyperbolic $3$-manifolds~\cite{garoufalidis20163d, garoufalidis20151}. It is however important to remark that only moves of these kinds that preserve a further geometric property, known as $1$-efficiency, do indeed preserve the $3$D index.

Note that topological invariance (\emph{i.e.}~independence from the triangulation) of the $3$D index is still unknown in general; it was however proved in~\cite{garoufalidis2019meromorphic}  that the $3$D index takes the same value on two triangulations of the same cusped manifold admitting a strict angle structure.

\subsection{The Hahn-Exton function}\label{ssec:hahnexton}

Bessel functions are well-known functions with multiple applications to physics, number theory and probability, to mention a few (see \textit{e.g.}~\cite{koelink1995q}, \cite{koelink1994quantum}, \cite{swarttouw1992addition}). Their $q$-analogues first appeared in a series of papers by F.H.~Jackson in 1903-1905.  We refer to Swarttouw's PhD thesis~\cite{SwarttouwPhD} for a recap on the history of these $q$-analogues. The same thesis contains many of the relations and tools we are going to use in what follows.
One of the most studied $q$-analogues of the Bessel function is known as the Hahn-Exton  $q$-Bessel function. This is defined for $\nu, z \in \C$ (except for $z =0$ whenever $Re\left(\nu\right)<0$ or if $Re\left(\nu\right) = 0$ and $\nu \neq 0$) as
\begin{equation}\label{eqn:def_J}
J_\nu \left(z;q\right) = z^\nu \frac{\left(q^{\nu +1}\right)_\infty}{\left(q\right)_\infty} \sum_{n\ge 0} \frac{\left(-1\right)^n q^{\binom{n+1}{2}}}{\left(q^{\nu+1}\right)_n \left(q\right)_n} z^{2k} = z^\nu \frac{\left(q^{\nu +1}\right)_\infty}{\left(q\right)_\infty} {_1 \phi_1}\left[\begin{matrix} 0  \\ q^{\nu+1}  \end{matrix} ; q,z^2 q \right],
\end{equation}
whenever $\nu >-1$. It is easy to prove (\emph{cf}.~Cor.~2~\cite{SwarttouwPhD}) that $J_\nu\left(z;q\right)$ is an analytic function in $z$ (except for $z=0$ whenever $\nu \in \Z$), and is an analytic function of $\nu$.\\

Note that, by \cite[Thm.~3.2]{SwarttouwPhD}, if $\nu \in \Z_{\le 0}$, it is possible to write:
$$J_{\nu}\left(z;q\right) = \left(-\sqrt{q}\right)^{-\nu} J_{-\nu}\left(q^{-\frac{\nu}{2}} z;q\right).$$

The main point of this section is to build on the following easy observation identifying the tetrahedral index with an evaluation of Hahn-Exton $q$-Bessel function, and deduce a whole suite of relations for $\tet$. This connection is well-known to experts in the field, but to the best of the author's knowledge, it never explicitly appeared in print.

\begin{prop}\label{prop:equivalence}
Let $m,e \in \Z$; then
$$\tet \left(m,e\right) = J_e \left(q^{-\frac{m}{2}};q\right).$$
\end{prop}
\begin{proof}
It is immediate to see that, if $e \ge 0$,
$$I_\Delta \left(m,e\right) = q^{-\frac{em}{2}}\sum_{n = e_+}^\infty \left(-1\right)^n \frac{q^{\binom{n}{2}} q^{n\left(1-m\right)}}{\left(q\right)_n \left(q\right)_{n+e}}. $$

Moreover, note that $\left(q\right)_{n+e} = \left(q\right)_n \left(q^{n+1}\right)_e = \left(q\right)_{e} \left(q^{e+1}\right)_n$.

Putting everything together, we get
$$I_\Delta \left(m,e\right) = \frac{q^{-\frac{em}{2}}}{\left(q\right)_e}\sum_{n = e_+}^\infty \left(-1\right)^n \frac{q^{\binom{n}{2}} q^{n\left(1-m\right)}}{\left(q\right)_n \left(q^{e+1}\right)_{n}}.$$
From this latter expression it is easy to realise that

\begin{equation}\label{eqn:index_qhyper_positive}
I_\Delta \left(m,e\right) = \frac{q^{-\frac{me}{2}}}{\left(q\right)_e} {_1 \phi_1}\left[\begin{matrix} 0  \\ q^{e+1}  \end{matrix} ; q,q^{1-m} \right].
\end{equation}

If instead $e<0$ a similar analysis yields
\begin{equation}\label{eqn:index_qhyper_negative}
I_\Delta \left(m,e\right) = \frac{q^{\frac{me}{2}}}{\left(q^e\right)_{-e}} {_1 \phi_1}\left[\begin{matrix} 0  \\ q^{-e+1}  \end{matrix} ; q,q^{1-e-m} \right],
\end{equation}
and we can conclude.
\end{proof}

In view of Proposition~\ref{prop:equivalence}, the relations satisfied by $\tet$ can be understood as consequences of more general relations known for $J_\nu\left(z;q\right)$; the $3$-term relations \eqref{eqn:index_3term_consecutive}, \eqref{eqn:index_3term_adjacent} are none other than the difference equations satisfied by $J_\nu$ (see \emph{e.g.}~\cite[4.3.1]{SwarttouwPhD}).

The quadratic identity in equation~\eqref{eqn:quadratic_identity} is implied by the orthogonality relations for $J_\nu$~\cite[Thm.~3.8]{SwarttouwPhD}; the pentagon identity can be derived straightforwardly by specialising a $q$-analogue of Graf's addition formula~\cite{koelink1995q}. The generating function for $\tet$ is paralleled in~\cite[Prop.~2.5]{koornwinder1992}.

\subsection{From $J_\nu$ to $\tet$}

The purpose of this section is to use Proposition~\ref{prop:equivalence} to translate certain known relations for $J_\nu$ into ones for $\tet$. \\

We start with \cite[Prop.~3.2]{groenevelt20183}, which (after a change of variable) gives
for $|t|<1$ and $\nu > -1$
\begin{equation}
\sum_{n = 0}^\infty q^{-\frac{\nu n}{2}} J_\nu \left(z q^\frac{n}2;q\right) \frac{t^n}{\left(q\right)_n} = z^{\nu} \frac{\left(q^{\nu+1}\right)_\infty}{\left(q\right)_\infty \left(t\right)_\infty} {_1 \phi_1}\left[\begin{matrix} t  \\ q^{\nu+1}  \end{matrix} ; q, q z^2 \right],
\end{equation}
which implies right away that for $|t|< 1$ and $e > -1$
\begin{equation}
\sum_{n = 0}^\infty I_\Delta \left(m-n, e\right)\frac{\left(t q^{-\frac{e}{2}}\right)^n}{\left(q\right)_n} = q^{-\frac{me}{2}} \frac{\left(q^{e+1}\right)_\infty}{\left(q\right)_\infty \left(t\right)_\infty} {_1 \phi_1}\left[\begin{matrix} t  \\ q^{e+1}  \end{matrix} ; q, q^{1-m}  \right].
\end{equation}

As a further consequence, by specialising to $t = q^{e+1}$ and using equation~\eqref{eqn:exp2}, we obtain the expression (again valid for $e \ge 0$)
\begin{equation*}
\sum_{n=0}^\infty \frac{\left( q^{\frac{e}2 +1}\right)^n}{\left(q\right)_n} \tet (m-n,e) = q^{-\frac{me}{2}} \frac{\left( q^{1-m}\right)_\infty}{\left( q\right)_\infty} = \begin{cases}
0 & \text{ if }m>0\\ \vspace{0.2cm}
\frac{q^{-\frac{me}{2}}}{\left( q\right)_{\infty}} & \text{ if }m=0\\  \vspace{0.2cm}
\frac{q^{-\frac{me}{2}}}{\left( q\right)_{-m}} & \text{ if }m<0.
\end{cases}
\end{equation*}

One interesting relation on tetrahedral indices can be deduced from \cite[Eqn~3.1]{swarttouw1992addition}; explicitly, for $e_1,e_2>-1$:
\begin{equation}
\begin{gathered}
\tet\left(m_1,e_1\right) \tet\left(m_2,e_2\right) = \\ q^{-\frac{m_1 e_1 + m_2 e_2}{2}}\frac{ \left(q^{e_1+1}\right)_\infty \left(q^{e_2+1}\right)_\infty}{\left(q\right)^2_\infty}  \sum_{n \ge 0} \frac{\left(-1\right)^n q^{-m_2 n} q^{\binom{n+1}{2}}}{ \left(q^{e_2+1}\right)_n \left(q\right)_n} {_2 \phi_1}\left[\begin{matrix} q^{-n}, q^{-n-e_2}  \\ q^{e_1 +1}  \end{matrix} ; q, q^{n-m_1 +m_2 +e_2 +1} \right].
\end{gathered}
\end{equation}

Yet another one comes from~\cite{koelink1994quantum}, and reads:
\begin{equation}
x^\nu J_\nu \left( \frac{z q^{\frac{\nu}{2}}}{x};q\right) = \sum_{n=0}^\infty \left(- \frac{z \sqrt{q}}{x^2} \right)^n \frac{\left(x^2\right)_n}{\left(q\right)_n} J_{\nu+n} \left(z q^{\frac{\nu+n}{2}};q\right)
\end{equation}
which, after applying equation~\eqref{eqn:tetrahedral_index_symmetries}, restricts to
\begin{equation}
I_\Delta \left(k,0\right) = \sum_{n = 0}^\infty q^{-\frac{n+k}{2}} I_\Delta \left(k+n,1\right) ,
\end{equation}
for $x = z = \sqrt{q}$. If instead we impose $z = q^{-k +\frac12} $, $x = \sqrt{q}$ we get
\begin{equation}\label{eqn:diagonal_term}
q^{\frac{k}{2}} I_\Delta \left(k,k\right) = \sum_{n = 0}^\infty \left(-q^{-k}\right)^n I_\Delta \left(k-n-1,k+n\right).
\end{equation}
~\\
A similar relation comes from the $q$-analogue of the multiplicative theorem for Bessel functions, proved in \cite[Thm.~3.11]{SwarttouwPhD}, and holds for $m<1$:
\begin{equation}
\tet\left(m - 2\lambda,e\right) = q^{\lambda e} \sum_{k \ge 0} q^{\left(1- \frac{m}{2}\right)k} \frac{\left(q^{2\lambda}\right)_k}{\left(q\right)_k} \tet\left(m,e+k\right)
\end{equation}

Finally for this section, we can translate~\cite[Cor.~2]{rubin2002toeplitz}
to get the following relation for $m_1,m_2,e \in \Z$ such that $m_1 +m_2\neq 0$:
\begin{equation}\label{eqn:1to3}
q^{\frac12 \left(m_1 - m_2\right)e} \tet\left(m_1+m_2,e\right) = q^{-e} \sum_{k,l \in \Z} q^{\frac{\left(2+ m_1 \right)l + \left(2- m_2\right) k}{2}} \tet \left(m_1,l\right) \tet \left(m_2,k\right) \tet \left(0,k+l-e\right).
\end{equation}

\begin{ques}
Given that the pentagon and quadratic identities~\eqref{eqn:quadratic_identity},\eqref{eqn:pentagon_identity} admit a straightforward interpretation in terms of Pachner moves, is there a geometric interpretation of any of the equations from this section?
\end{ques}

\subsection{From $\tet$ to $J_\nu$ }

The main result of~\cite{3dindexdehn} relied crucially on an identification between certain $q$-hypergeometric functions and a generating function for ``diagonal'' tetrahedral indices. More precisely, define
\begin{equation}\label{eqn:phi_expression}
\varphi_r\left(z,q\right) =
\begin{cases}
\left(-z q^{-\frac12}\right)^{\frac{r}{2}} {_3\phi_3} \left[\begin{matrix} -z^{-1} \sqrt{q} &  -z\sqrt{q}  & 0\\
-q & \sqrt{q} & -\sqrt{q}
 \end{matrix}
; q, q^{1-\frac{r}{2}} \right]   & r \text{ even}\\
~\\
\left(-z\right)^{\frac{r-1}{2}} q^{\frac{1-r}{2}}
\frac{\left(1+z\right)}{1-q} {_3\phi_3} \left[\begin{matrix} -z^{-1} q   & -z q  & 0\\
-q & q^\frac{3}{2} & -q^{\frac{3}{2}}
 \end{matrix}
; q, q^{\frac{3-r}{2}} \right]   & r \text{ odd.}
\end{cases}
\end{equation}
Then, by~\cite[Sec.~10]{3dindexdehn} we have
\begin{equation}\label{eqn:seriesindexshifted}
\varphi_r\left(z,q\right) = \sum_{e \in \Z}  I_\Delta \left(e-r,e\right) z^e = \sum_{e \in \Z} J_e \left(q^{\frac{r-e}{2}};q\right) z^e.
\end{equation}

It turns out that $\varphi_r\left(z,q\right)$ can be identified with a product of certain $q$-analogues of trigonometric functions, introduced by various authors (see the discussion in \cite[Sec.~2]{Nevanlinna}).

\begin{equation}\label{eqn:def_q_trigonometric}
\qsin \left(z\right) = \sum_{n\ge 0} \frac{\left(-1\right)^n q^{n\left(n+1\right)}}{\left(q\right)_{2n+1}} z^{2n+1}, \,\,\,\,\,\,\,\,\, \qcos \left(z\right) = \sum_{n\ge 0} \frac{\left(-1\right)^n q^{n^2}}{\left(q\right)_{2n}} z^{2n}.
\end{equation}

Equivalently,
\begin{equation}\label{eqn:trigonometric_as_hypergeometric}
\qsin\left(z\right) = \frac{z}{1-q} {_1\phi_1} \left[ \begin{matrix} 0\\ q^3 \end{matrix}; q^2,q^2 z^2 \right],\,\,\,\,\,\,\,\,\,\, \qcos\left(z\right) =  {_1\phi_1} \left[ \begin{matrix} 0\\ q \end{matrix}; q^2,q z^2\right].
\end{equation}

Therefore in particular
$$\qsin \left(z\right) = \frac{\left(q^2;q^2\right)_\infty }{\left(q;q^2\right)_\infty} z^{\frac12} J_{\frac12} \left(z;q^2\right),\,\,\,\,\,\,\,\,\, \qcos \left(z\right) = \frac{\left(q^2;q^2\right)_\infty }{\left(q;q^2\right)_\infty} z^{\frac12} J_{-\frac12} \left(q^{-\frac12}z;q^2\right).$$

\begin{rmk}
It is also possible to express these $q$-sine and $q$-cosines as basic hypergeometric functions with base $q$:
$$\qsin \left(z\right) = \frac{z}{1-q} {_2\phi_3} \left[ \begin{matrix} 0 &0\\ -q & q^{\frac32} & -q^{\frac32} \end{matrix}; q, -\left(z q\right)^2\right]$$
$$\qcos \left(z\right) = {_2\phi_3} \left[ \begin{matrix} 0 &0\\-q  & q^{\frac12} & -q^{\frac12} \end{matrix}; q,-\left(q^{\frac12} z\right)^2\right]$$
\end{rmk}

As a straightforward consequence of \cite[Prop.~1]{Nevanlinna}, we proved (\cite[Sec.~10]{3dindexdehn}) that, if $r$ is even, we get
\begin{equation}\label{eqn:phi_trigonometric_even}
\varphi_r \left(z,q\right) = \left(-z q^{-\frac12}\right)^{\frac{r}{2}} \left[\qcos \left(z^{\frac12} q^{\frac{1-r}{4}}\right) \qcos\left(- z^{-\frac12} q^{\frac{1-r}{4}}\right) + q^{\frac12} \qsin\left(z^{\frac12} q^{\frac{1-r}{4}}\right)\qsin\left(- z^{-\frac12} q^{\frac{1-r}{4}}\right) \right],
\end{equation}
while for $r$ odd
\begin{equation}\label{eqn:phi_trigonometric_odd}
\varphi_r \left(z,q\right) = \left(-1\right)^{\frac{r-1}{2}} q^{\frac{1-r}{4}} z^{\frac{r}{2}}\left[ \qsin\left(z^{\frac12} q^{\frac{1-r}{4}} \right) \qcos\left(-z^{-\frac12} q^{\frac{1-r}{4}}\right) - \qcos\left(z^{\frac12} q^{\frac{1-r}{4}}\right) \qsin\left(-z^{-\frac12} q^{\frac{1-r}{4}} \right)\right]
\end{equation}

The previous identifications yield (among others) the following (possibly new) identities (\textit{cf.}~\cite[Cor.~10.3]{3dindexdehn}); if $r \in \Z$ is even, then:
\begin{align}
\sum_{\nu \in \Z} \left(-\sqrt{q}\right)^\nu  J_{\nu}\left(q^{\frac{r-\nu}{2}};q\right) & =1 ,\\
\sum_{\nu \in \Z} \left(-\sqrt{q}\right)^{-\nu}  J_{\nu}\left(q^{\frac{r-\nu}{2}};q\right)& = q^{-\frac{r}{2}} .
\end{align}
If instead $r \in \Z$ is odd:
\begin{align}
\sum_{\nu \in \Z} \left(-q\right)^\nu J_{\nu}\left(q^{\frac{r-\nu}{2}};q\right)  & =1 ,\\
\sum_{\nu \in \Z} \left(-q\right)^{-\nu} J_{\nu}\left(q^{\frac{r-\nu}{2}};q\right) & = -q^{-r} .
\end{align}

\section{$3$D index and Dehn fillings}\label{sec:3Dindex}

To simplify and streamline what follows, we are going to restrict ourselves to the special case of the complement of a link in $S^3$ with $r \ge 1$ components.
We refer to~\cite{thurston2014three, benedetti1992lectures, purcell2020hyperbolic} for basic notions in hyperbolic geometry and knot theory.

Such a manifold $M$ can be constructed and concisely described by an \emph{ideal triangulation} $\mathcal{T}$. This is a triangulation of $M$ built by ideal tetrahedra, \textit{i.e.}~compact tetrahedra with the vertices removed. A given cusped 3-manifold admits infinitely many different ideal triangulations, all related by a finite set of combinatorial moves known as Pachner moves (see \textit{e.g.}~\cite{matveev2007algorithmic,rubinstein2020traversing}).

\begin{figure}[h]
\centering
\includegraphics[width=11cm]{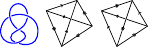}
\caption{The complement in $S^3$ of the figure-eight knot $4_1$ shown on the left, can be obtained by gluing two ideal tetrahedra, using the identifications displayed on the edges (\textit{cf.}~\cite{Thu78notes}). }
\label{fig:4_1}
\end{figure}

Thurston~\cite{Thu78notes} introduced an important method for constructing hyperbolic geometry on a cusped 3-manifold $M$ by gluing together geometric ideal tetrahedra in hyperbolic $3$-space $\mathbb{H}^3$. (These are  given as the convex hull of four generic points in $\partial_\infty \mathbb{H}^3$.)
Up to isometries of $\mathbb{H}^3$ such an ideal tetrahedron is completely determined by its \emph{shape parameter}. This is a complex number $z$ such that $\mathrm{Im}\left(z\right) >0$. The geometry on $M$ is then obtained by gluing together $n$ ideal tetrahedra along their faces via isometries. Thurston's gluing equations are a set of equations in the shape parameters of the triangulation's tetrahedra that guarantee that tetrahedra fit together correctly around the edges to give geometry on $M$ locally modelled on $\mathbb{H}^3$.
Importantly, the (logarithm of) these equations, expressed in matrix form, can be easily recovered by using the freely available program SnapPy~\cite{snappy}, via the \texttt{gluing\_equations()} command. \\

There are several approaches to compute the $3$D index; either in terms of \textit{edge weights} (see \textit{e.g.}~\cite{garoufalidis20151}) or using  \textit{normal surfaces}~\cite{garoufalidis20163d}. For a detailed explanation of the subtleties involved in the computations we refer to these papers as well as Section~14 and Appendix C in~\cite{3dindexdehn}.

We now give a brief description of the former approach.
We start with the gluing matrix $G_{\T}$ for the ideal triangulation $\T$, with $n$ tetrahedra $\Delta_1, \ldots,\Delta_n$,
and a set of $n$ edges denoted by $\E$. Roughly speaking, $I_\T^\omega\left(q\right)$ is an infinite sum of products of tetrahedral indices over integer weights assigned to the edges of $\T$, defined using the combinatorics of $\T$.

Consider weight functions $k : \E \to \Z$ assigning an integer weight $k\left(e\right)$ to each edge class $e\in \E$; this gives a weight on each edge of each tetrahedron $\Delta_j$, $j=1, \ldots, n$.  Adding up these weights on each pair of opposite edges in $\Delta_j$ yields a vector of {\em quad weights} $q_j\left(k\right) =\left(a_j\left(k\right),b_j\left(k\right),c_j\left(k\right)\right) \in \Z^3$.
Then define
\begin{equation}
I_\T^{0}\left(q\right)=  \sum_{\substack{k : \E \to \Z, \\ k|_{\E'}=0}}  q^{ k\left(\E\right) } \, \prod_{j=1}^n \left(-q^\frac12\right)^{-b_j\left(k\right)}\tet \left(b_j\left(k\right)-c_j\left(k\right), a_j\left(k\right)-b_j\left(k\right)\right)
\end{equation}
where $k\left(\E\right)=\sum_{e\in\E }  k\left(e\right)$ and $\E'$ is a suitable choice of $r$ edges. See~\cite[Sec.~4.6]{garoufalidis20151} to understand which choices are admissible (for $r=1$ any edge is suitable).

In general,
\begin{equation}
I_\T^\omega\left(q\right) =  \sum_{\substack{k : \E \to \Z, \\ k|_{\E'}=0}}  q^{ k\left(\E\right) } \, \prod_{j=1}^n \left(-q^\frac12\right)^{-b_j^\omega\left(k\right)}\tet \left(b_j^\omega\left(k\right) -c_j^\omega\left(k\right), a_j^\omega\left(k\right)-b_j^\omega\left(k\right)\right)
\end{equation}
where $\left(a_j^\omega\left(k\right),b_j^\omega\left(k\right),c_j^\omega\left(k\right)\right)  \in \Z^3$ depend on $k$ and some \textit{boundary conditions} $\omega \in \mathrm{H}_1\left(\bd M;\Z\right)$.
All the coefficients in these expressions can be read off easily from the gluing equations for $\T$ given by SnapPy.
With these conventions, $I_\T^\omega\left(q\right)$ is defined for
$$\omega \in \mathcal{K} = \ker\left( \mathrm{H}_1\left(\bd M ;\Z\right) \to \mathrm{H}_1\left(M; \Z_2\right)\right).$$

\begin{rmk}
It was proved in~\cite{garoufalidis20163d} that the index $I_\T^\omega\left(q\right)$ converges for all boundary conditions $\omega$ if and only if the ideal triangulation $\T$ is $1$-efficient. Here $1$-efficiency means that that $\T$ does not
contain any embedded closed normal surfaces of non-negative Euler characteristic, except
for vertex-linking tori.
\end{rmk}

\begin{example}
For Thurston's ideal triangulation $\T$ of the figure eight knot complement $S^3 \setminus 4_1$ from Figure~\ref{fig:4_1}
with two tetrahedra, and the standard meridian and longitude $\mu,\lambda \in H_1\left(\bd M;\Z\right)$, $\mathcal{K}$ is spanned by $\{2\mu,\lambda\}$. SnapPy produces the following gluing matrix for $\T$; we exclude the edge $e_2$ by giving it weight $0$.

 \begin{table}[h]
\begin{tabular}{|c|c||c|c|c|c|c|c|}
\hline
 edge/curve&weight& ~~$a_1$~~ & ~~$b_1$~~ & ~~$c_1$~~ & ~~$a_2$~~ & ~~$b_2$~~ & ~~$c_2$ \\
\hline
\hline
$e_1$ &$k$ &2 & 1 & 0 & 2 &1 & 0 \\
\hline
$e_2$ &0&0& 1 & 2& 0 & 1 & 2 \\
\hline
$2\mu$ & $x$&0& 0 & 1& -1 & 0 & 0 \\
\hline
$2 \lambda$ &$\frac{1}{2}y$ &0& 0 & 0& 2  & 0 & -2\\
\hline
\end{tabular}
\vspace{0.2cm}
\label{table:Fig8Gluing}
\caption{Gluing data for figure eight knot complement.}
\end{table}

Then,
\begin{equation}\label{eqn:index4_1}
I_{\T}^{2x \mu + y\lambda}\left(q\right)
 = \sum_{k \in \Z} \, \tet\left(k-x,k\right) \tet\left(k+y,k-x+y\right).
\end{equation}

For example,
\begin{gather*}
\qquad I_{\T}^0\left(q\right) =
1-2 q-3 q^2+2 q^3+8 q^4+18 q^5+18 q^6+14 q^7-12 q^8 + \ldots \\
I_{\T}^{2\mu}\left(q\right) =
-2 q-2 q^2+2 q^3+8 q^4+16 q^5+16 q^6+10 q^7-14 q^8+ \ldots \\
I_{\T}^\lambda\left(q\right) = -2 q^{3/2}+4 q^{7/2}+10 q^{9/2}
+14 q^{11/2}+10 q^{13/2}-2 q^{15/2}-32 q^{17/2} + \ldots \\
\end{gather*}
\end{example}

The main result from~\cite{3dindexdehn}  gives a rigorous proof of a formula by Gang and Yonekura~\cite{gang2018symmetry}, expressing the $3$D index of manifolds obtained as Dehn filling in terms of the $3$D index of the unfilled manifold.

Here we only consider Dehn fillings $M\left(\alpha\right)$ along slopes $\alpha$ on one
torus boundary component $T \subseteq \partial M$.
Then the Gang-Yonekura formula~\cite{gang2018symmetry} gives the $3$D-index for $M\left(\alpha\right)$ as an infinite linear combination of the $3$D-indices $I_M^\gamma$ for $M$ over boundary classes $\gamma$
having intersection number $0$ or $\pm 2$ with $\alpha$:
\begin{equation}\label{eqn:gang}
I_{M(\alpha)} = \frac{1}{2} \left( \,\sum_{\substack{\gamma \in \mathcal{K}_T\\ \gamma \cdot \alpha = 0}} \left(-1\right)^{|\gamma|}\left(q^{\frac{1}{2}|\gamma|} +q^{-\frac{1}{2}|\gamma|}\right) I_M^\gamma\left(q\right)  -\sum_{\substack{\gamma \in \mathcal{K}_T\\ \gamma \cdot \alpha = \pm 2}} I_M^\gamma (q) \right),
\end{equation}
where $|\gamma|$ is the number of components of $\gamma$ and $\mathcal{K}_T= \ker\left( \mathrm{H}_1\left(T ;\Z\right) \to \mathrm{H}_1\left(M; \Z_2\right)\right)$.
(For simplicity, we assume the boundary data vanishes on all components of $\partial M \setminus T$.)\\

The key takeaway come from the fact that every cusped $3$-manifold can be described as Dehn surgery on infinitely many links in $S^3$ (see \textit{e.g.}~Figure~\ref{fig:equivalent_knots}), and each such description is expected to produce the same $3$D index (\textit{cf.}~the main theorem in \cite{3dindexdehn}), thus giving non-trivial relations between the corresponding $q$-series.

\begin{figure}
\centering
\includegraphics[width=0.7\linewidth]{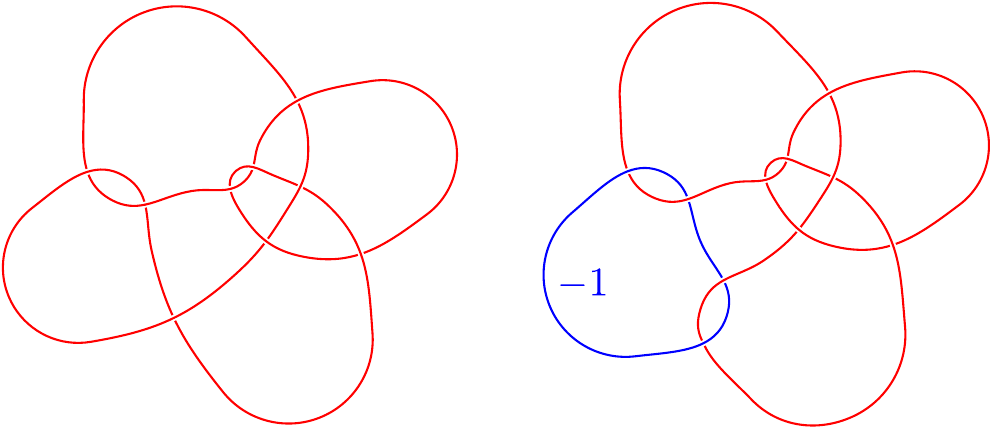}
\caption{The complement of the knot $8_{16}$ (shown on the left) can be alternatively obtained \textit{e.g.}~by performing $(-1)$-framed Dehn surgery on the blue component of the link shown on the right. (Figure created using KLO~\cite{KLO})}
\label{fig:equivalent_knots}
\end{figure}

\section{Conjectural relations in the closed case}

There is currently no precise mathematical definition of $3$D index for closed $3$-manifolds. In spite of this, equation~\eqref{eqn:gang} can be applied to knot complements. Remarkably, experimental applications of the formula appear to yield meaningful results (\textit{cf.}~\cite[Sec.~13]{3dindexdehn}). As an example, in the case of hyperbolic knots, it appears that computing the index for exceptional slopes (\emph{i.e.}~producing non-hyperbolic manifolds, see \textit{e.g.}~\cite{dunfield2020census}) always gives a somewhat trivial or diverging index.
In certain special cases, such as surgeries on alternating torus knots (\emph{cf}.~\cite[Sec.~13.1]{3dindexdehn}), whose index is particularly simple to start with, it is possible to give an explicit answer for all surgeries.

\begin{example}
As shown in~\cite[Sec.~13.2]{3dindexdehn},
combining the expression from~\eqref{eqn:index4_1} and the definition from equation~\eqref{eqn:phi_expression}, it is straightforward to show that the $3$D index $I_\mathcal{T}^{2x\mu+ y\lambda} \left(q\right)$ of the complement of $4_1$ can be computed as the coefficient of $z^0$ in the expression:
$$z^y \varphi_x \left(z\right) \varphi_{-x} \left(z\right).$$

This knot is known to admit exactly $10$ exceptional surgeries, given by the slopes
$E\left(4_1\right) = \left\{\frac{1}{0}, \frac01, \pm\frac11, \pm\frac21, \pm \frac31, \pm\frac41\right\}$.
See Table~\ref{tab:n1_fig8} below for a list of the index's values for some integer
surgeries on $4_1$, in $q$-degrees $\le 10$.

\begin{ques}
The only entry in Table~\ref{tab:n1_fig8} which is known to contain an exact result (following~\cite[Thm.~13.2]{3dindexdehn}) is $(0,1)$-surgery on the $4_1$'s complement. Is it possible to prove the other identities (corresponding to the other exceptional surgeries for $4_1$) suggested by the table?
\end{ques}
\begin{table}
\begin{tabular}{|c|c|}
\hline
Slope & $3$D index\\
\hline
\multirow{2}{*}{$\left(1,0\right)$} & $ O\left(q^{11}\right)$ \\
 & \\
\hline
\multirow{2}{*}{$\left(0,1\right)$} & $1$ \\
 & \\
\hline
\multirow{2}{*}{$\left(1,1\right)$} & $1 + O\left(q^{11}\right)$ \\
 & \\
\hline
\multirow{2}{*}{$\left(2,1\right)$} & $1 + O\left(q^{11}\right)$ \\
 & \\
\hline
\multirow{2}{*}{$\left(3,1\right)$} & $1 + O\left(q^{11}\right)$ \\
 & \\
\hline
\multirow{2}{*}{$\left(4,1\right)$} & $\infty - 2q^2 - 2q^3 - 4q^4 - 4q^5  $ \\
 & $- 6q^6 - 4q^7 - 8q^8 - 6q^9 - 8q^{10} + O\left(q^{11}\right)$ \\
\hline
\multirow{2}{*}{$\left(5,1\right)$} & $1 - q - 2 q^2 - q^3 - q^4 + q^5 + 2 q^6 + 7 q^7$ \\
 & $+ 8 q^8 + 12 q^9 + 14 q^{10} + O\left(q^{11}\right)$ \\
\hline
\multirow{2}{*}{$\left(6,1\right)$} & $1 - t - q^{\frac32} - q^2 - q^\frac52 + q^\frac92 + 3 q^5 + 4 q^{\frac{11}{2}} + 4 q^6+ 6 q^{\frac{13}{2}}$ \\
 & $ + 8 q^7 + 7 q^\frac{15}2 + 7 q^8 + 9 q^\frac{17}2 + 8 q^9 + 6 q^\frac{19}2 + 4 q^{10} + O\left(q^{11}\right)$ \\
\hline
\multirow{2}{*}{$\left(7,1\right)$} & $1 - q^2 + q^4 + 3 q^5 + 3 q^6 + 6 q^7$ \\
 & $+ 4 q^8 + 2 q^9 - 4 q^{10} + O\left(q^{11}\right)$ \\
\hline
\multirow{2}{*}{$\left(8,1\right)$} & $1 - q^2 + 2 q^4 + 5 q^5 + 6 q^6 + 8 q^7$ \\
 & $+ 4 q^8 - 2 q^9 - 14 q^{10} + O\left(q^{11}\right)$ \\
\hline
\multirow{2}{*}{$\left(9,1\right)$} & $1 - q^2 + 2 q^5 + 2 q^6 + 4 q^7 + q^8$ \\
  & $- q^9 - 6 q^{10} + O\left(q^{11}\right)$ \\
\hline
\multirow{2}{*}{$\left(10,1\right)$} & $1 - q^2 + q^\frac52 + q^\frac72 + q^\frac92 + 2 q^5 + 2 q^\frac{11}2 + 2 q^6 + q^\frac{13}2 $ \\
  & $+ 4 q^7- q^\frac{15}2 + 2 q^8 - 5 q^\frac{17}2 - 9 q^\frac{19}2 - 4 q^{10} + O\left(q^{11}\right)$ \\
\hline
\end{tabular}
\caption{Table from~\cite{3dindexdehn}, reporting the computations for integer-sloped surgeries on the complement of the figure eight knot.}
\label{tab:n1_fig8}
\end{table}
\end{example}

The number of exceptional surgeries on hyperbolic knots in $S^3$ is known to be finite. In fact, the number of exceptions is at most $10$~\cite{lackenby2013maximal}, and $4_1$ is the only known
knot with the maximal possible number (see \textit{e.g.} the survey~\cite{gordon2012exceptional}). Given all this, we can ask the following question (see also~\cite[Sec.~15]{3dindexdehn}):
\begin{ques}
Given a triangulation of a cusped $3$-manifold, is it possible to detect all exceptional surgeries using only ``$q$-hypergeometric'' (\textit{i.e.}~not based on geometric or topological considerations) techniques?
\end{ques}

\end{document}